\documentclass[final,1p,times]{elsarticle}
\usepackage{amsmath}
\usepackage{amssymb}
\usepackage{url}
\usepackage{amsfonts}
\usepackage{graphicx}
\usepackage{rotating}
\usepackage{floatflt,epsfig}
\usepackage{lineno,hyperref}
\usepackage{enumerate}
\usepackage{colortbl}
\usepackage{array,tabularx,tabulary,booktabs}
\usepackage{longtable}
\usepackage{multirow}
\usepackage{wrapfig}
\usepackage{subcaption}
\newcolumntype{^}{>{\currentrowstyle}}

\makeatletter
\def\ps@pprintTitle{%
   \let\@oddhead\@empty
   \let\@evenhead\@empty
   \let\@oddfoot\@empty
   \let\@evenfoot\@oddfoot
}
\makeatother

\newtheorem{lemma}{Lemma}

\newtheorem{theorem}{Theorem}

\newtheorem{proposition}{Proposition}

\newtheorem{construction}{Construction}

\begin{document}
\renewcommand{\abstractname}{Abstract}
\renewcommand{\refname}{References}
\renewcommand{\tablename}{Table}
\renewcommand{\arraystretch}{0.9}
\thispagestyle{empty}
\sloppy

\begin{frontmatter}
\title{Vertex connectivity of some classes of divisible design graphs}

\author[01,02]{Dmitry~Panasenko}
\ead{makare95@mail.ru}

\address[01]{Chelyabinsk State University, Brat'ev Kashirinyh st. 129, Chelyabinsk, 454021, Russia}

\address[02]{Krasovskii Institute of Mathematics and Mechanics, S. Kovalevskaja st. 16, Yekaterinburg, 620990, Russia}

\begin{abstract}
A $k$-regular graph is called a divisible design graph if its vertex set can be partitioned into $m$ classes of size $n$, such that two distinct vertices from the same class have exactly $\lambda_1$ common neighbours, and two vertices from different classes have exactly $\lambda_2$ common neighbours. In this paper, we find the vertex connectivity of some classes of divisible design graphs, in particular, we present examples of divisible design graphs, whose vertex connectivity is less than $k$, where $k$ is the degree of a vertex. We also show that the vertex connectivity of one series of divisible design graphs may differ from k by any power of 2.
\end{abstract}

\begin{keyword} Deza graph \sep divisible design graph \sep strongly regular graph \sep vertex connectivity
\vspace{\baselineskip}
\MSC[2010] 05C50\sep 05E10\sep 15A18
\end{keyword}
\end{frontmatter}

\section{Introduction}\label{sect1}
Deza graphs were introduced in~\cite{EFHHH1999} as a generalisation of strongly regular graphs. A \emph{strongly regular graph} (SRG for short) $G$ with parameters $(v,k,\lambda,\mu)$ is a $k$-regular graph with $v$~vertices such that any two adjacent vertices have $\lambda$ common neighbours and any two non-adjacent vertices have $\mu$ common neighbours. A \emph{Deza graph} $\Gamma$ with parameters $(v,k,b,a)$ is a $k$-regular graph with $v$ vertices for which the number of common neighbours of two distinct vertices takes just two values, $b$ or $a$, where $b \geqslant a$. A Deza graph of diameter 2 that is not a strongly regular graph is called a \emph{strictly Deza graph}.

A $k$-regular graph is called a \emph{divisible design graph} (DDG for short) if its vertex set can be partitioned into $m$ classes of size $n$, such that two distinct vertices from the same class have exactly $\lambda_1$ common neighbours, and two vertices from different classes have exactly $\lambda_2$ common neighbours. The definition implies that all divisible design graphs are Deza graphs. Divisible design graphs were first studied in master's thesis by M.A.~Meulenberg~\cite{M2008} and then studied in more detail in~\cite{CH2014,HKM2011}.

The vertex connectivity of a $k$-regular Cayley graph is at least $\frac{2}{3}(k+1)$ (see~\cite{GR2001}). The vertex connectivity of a strongly regular graph is equal to its valency, as was proved by A.E.~Brouwer and D.M.~Mesner in~\cite{BM1985}. In~\cite{BK2009} the same result was obtained in general for distance-regular graphs.

The vertex connectivity of Deza graphs obtained from strongly regular graphs by dual Seidel switching was studied in~\cite{GGK2014} by A.L.~Gavrilyuk, S.~Goryainov and V.V.~Kabanov and in~\cite{GP2019} by S.~Goryainov and D.~Panasenko. As a result of these studies, an infinite series of Deza graphs with vertex connectivity $k - 1$ was found. 

In this paper we find the vertex connectivity of some classes of divisible design graphs. We focus on the cases when the vertex connectivity is less than $k$. In particular, we present DDGs with vertex connectivity $k - 1$ and, for any positive integer $t$, we present a DDG whose vertex connectivity equals $k - 2^t$ $(k > 2^t)$.

This paper is organised as follows. In Section~\ref{sect2}, we give some definitions, notations and preliminary results on SRGs, DDGs and vertex connectivity. In Section~\ref{sect3}, we present results on vertex connectivity. In Section~\ref{sect4}, we discuss the problem of finding the vertex connectivity of DDGs in general.

\section{Preliminaries}\label{sect2}

A strongly regular graph $G$ is called \emph{primitive} if both $G$ and its complement are connected. 

\begin{lemma}[{\cite[Theorem 1.3.1]{BCN1989}}]\label{teosrg}
Let $G$ be a primitive strongly regular graph with parameters $(v,k,\lambda,\mu)$. Then the following statements hold.

{\rm (1)} $G$ has three distinct eigenvalues $k,r,s$, where $k>r>0>s$. Moreover, $r$~and~$s$ satisfy the quadratic equation $x^2+(\mu-\lambda)x+(\mu-k)=0.$

{\rm (2)} If the eigenvalues~$r$ and~$s$ have equal multiplicities, then $r=(-1+\sqrt{v})/2$ and ~~~~~~~~~~~~~~~~~~~~~~~~~~~~~~~~~~ $s=(-1-\sqrt{v})/2$. Otherwise, $r$ and $s$ are integers.

{\rm (3)} The equalities $\mu=k+rs$ and $\lambda=\mu+r+s$ hold.
\end{lemma}

A strongly regular graph with $s = -2$ is called a \emph{Seidel graph}. These graphs were characterised in \cite[Theorem 3.12.4]{BCN1989}.

\medskip
An incidence structure with $v$ points and $v$ blocks of constant size $k$ is called a \emph{symmetric $2$-$(v, k, \lambda)$-design} if any pair of points occur together in exactly $\lambda$~blocks and any two blocks intersect in exactly $\lambda$ points. 

An incidence structure on $v$ points with constant block size $k$ is a \emph{(group) divisible design} whenever the point set can be partitioned into $m$ classes of size $n$, such that two points from one class occur together in exactly $\lambda_1$ blocks, and two points from different classes occur together in exactly $\lambda_2$ blocks. A divisible design $D$ is called \emph{symmetric} if the dual of $D$ (that is, the design with the transposed incidence matrix) is again a divisible design with the same parameters as~$D$. Equivalently, DDGs can be defined as graphs whose adjacency matrix is the incidence matrix of a symmetric divisible design. A DDG with $m = 1$, $n = 1$ or $\lambda_1 = \lambda_2$ is called \emph{improper} (these DDGs are $(v,k,\lambda)$-graphs), otherwise it is called \emph{proper}. A $(v,k,\lambda)$-graph is a $k$-regular graph on $v$ vertices with the property that any two distinct vertices have exactly $\lambda$ common neighbours, that is, a strongly regular graph with $\lambda = \mu$, a~clique or a coclique.

For a positive integer $t$, denote by $I_t$, $O_t$ and $J_t$ the identity matrix, the zero matrix and the all-ones matrix of size $t \times t$, respectively. For positive integers $m$ and $n$, denote by $K_{(m,n)}$ the matrix $I_m \otimes J_n$. Note that $K_{(m,n)} = diag(J_n, \ldots, J_n)$. A graph $\Gamma$ is a DDG with parameters $(v, k, \lambda_1, \lambda_2, m, n)$ if and only if its adjacency matrix $A$ satisfies
$$A^2 = kI_v + \lambda_1(K_{(m,n)} - I_v) + \lambda_2(J_v - K_{(m,n)}).$$

The formula for $A^2$ also gives strong information about the eigenvalues of $A$ and their multiplicities (see the following two lemmas).

\begin{lemma}[{\cite[Lemma 2.1]{HKM2011}}]\label{l1}
$A$ has at most five distinct eigenvalues $k$, $\sqrt{k - \lambda_1}$, $-\sqrt{k - \lambda_1}$, $\sqrt{k^2 - \lambda_2v}$ and $-\sqrt{k^2 - \lambda_2v}$ with corresponding multiplicities $1, f_1, f_2, g_1$ and $g_2$, where $f_1 + f_2 = m(n - 1)$ and $g_1 + g_2 = m - 1$.
\end{lemma}

\begin{lemma}[{\cite[Theorem 2.2]{HKM2011}}]\label{l2}
Consider a proper DDG with parameters $(v, k, \lambda_1, \lambda_2,$ $m, n)$ and eigenvalue multiplicities $f_1, f_2, g_1, g_2$. Then:

{\rm (1)} $k - \lambda_1$ or $k^2 - \lambda_2v$ is a nonzero square;

{\rm (2)} if $k - \lambda_1$ is not a square, then $f_1 = f_2 =m(n -1)/2$;

{\rm (3)} if $k^2 - \lambda_2v$ is not a square, then $g_1 = g_2 = (m -1)/2$.
\end{lemma}

Let $V_1 \cup \ldots \cup V_t$ be a partition of the vertex set of a graph $\Gamma$ with the property that, for any $i,j \in \{1, \ldots, t\}$, every vertex of $V_i$ has exactly $r_{ij}$ neighbours in $V_j$ for some constant $r_{ij}$ (depending on $i$ and $j$). Then $V_1 \cup \ldots \cup V_t$ is called an \emph{equitable $t$-partition} of $\Gamma$. The matrix $R = (r_{ij})_{t \times t}$ is called the \emph{quotient matrix} of the equitable partition.

The vertex partition from the definition of a DDG is called the \emph{canonical partition}.

\begin{lemma}[{\cite[Theorem 3.1]{HKM2011}}]\label{l3}
The canonical partition of the vertex set of a DDG is equitable, and the quotient matrix $R$ satisfies $$R^2 = RR^T = (k^2 - \lambda_2v)I_m +\lambda_2nJ_m.$$
Moreover, the eigenvalues of $R$ are $k$, $\sqrt{k^2 - \lambda_2v}$ and $-\sqrt{k^2 - \lambda_2v}$ with corresponding multiplicities $1$, $g_1$ and $g_2$, where $g_1$ and $g_2$ are given by Lemmas~\ref{l1} and \ref{l2}.
\end{lemma}

A graph is called \emph{walk-regular}, whenever for every $l \geqslant 2$ the number of closed walks of length $l$ at a vertex $x$ is independent of the choice of $x$.

\begin{lemma}[{\cite[Theorem 4.3]{CH2014}}]\label{lwr}
A proper DDG is walk-regular if and only if the quotient matrix $R$ has constant diagonal.
\end{lemma}

\subsection{Constructions of DDGs}\label{sect2_1}

The \emph{incidence graph} of a design with incidence matrix $N$ is the bipartite graph with adjacency matrix
$\begin{bmatrix}
O & N\\
N^\top & O
\end{bmatrix}$.

\begin{construction}[{\cite[Construction 4.1]{HKM2011}}]\label{c1}
The incidence graph of a symmetric $2$-$(n,k,\lambda_1)$-design with $1 < k \leqslant n$ is a proper DDG with parameters $(2n, k, \lambda_1, \lambda_2, 2, n)$, where $\lambda_2 = 0$.
\end{construction}

\begin{proposition}[{\cite[Proposition 4.3]{HKM2011}}]\label{p1}
For a proper connected DDG $\Gamma$ with parameters $(v, k, \lambda_1, \lambda_2, m, n)$, the following statements are equivalent.

{\rm (1)} $\lambda_2 = 0$.

{\rm (2)} $\Gamma$ comes from Construction 1.
\end{proposition}

\begin{construction}[{\cite[Construction 4.4]{HKM2011}}]\label{c2}
Let $A'$ be the adjacency matrix of a connected $(m,k',\lambda')$-graph with $1 < k' < m$. Then, for any positive integer $n$, $n > 1$, the matrix $A' \otimes J_n$ is the adjacency matrix of a proper DDG with parameters $(mn, k, \lambda_1, \lambda_2, m, n)$, where $k = \lambda_1 = nk'$ and $\lambda_2 = n\lambda'$.
\end{construction}

\begin{proposition}[{\cite[Proposition 4.5]{HKM2011}}]\label{p2}
For a proper connected DDG $\Gamma$ with parameters $(v, k, \lambda_1, \lambda_2, m, n)$, the following statements are equivalent.

{\rm (1)} $\lambda_1 = k$.

{\rm (2)} $\Gamma$ comes from Construction 2.
\end{proposition}

\begin{construction}[{\cite[Construction 4.6]{HKM2011}}]\label{c3}
Let $A_1, \ldots, A_m$ $(m \geqslant 2)$ be the adjacency matrices of not necessary connected $(n,k',\lambda')$-graphs $($possibly non-isomorphic, but having the same parameters$)$ with $0 \leqslant k' \leqslant n-2$. Then the matrix $J_v - K_{(m,n)} + diag(A_1, \ldots, A_m)$ is the adjacency matrix of a proper DDG with parameters $(mn, k,$ $\lambda_1,$ $\lambda_2, m, n)$, where $k = k' + n(m-1)$, $\lambda_1 = \lambda' + n(m-1)$ and $\lambda_2 = 2k - v$.
\end{construction}

\begin{proposition}[{\cite[Proposition 4.7]{HKM2011}}]\label{p3}
For a proper DDG $\Gamma$ with parameters $(v, k, \lambda_1, \lambda_2, m, n)$, the following statements are equivalent.

{\rm (1)} $\lambda_2 = 2k - v$.

{\rm (2)} $\Gamma$ comes from Construction 3.
\end{proposition}

The \emph{lexicographic product} or \emph{graph composition} $G[H]$ of graphs $G$ and $H$ is the graph with vertex set $V(G) \times V(H)$ and adjacency defined by
$$(u_1, v_1) \sim (u_2, v_2) \text{ if and only if } u_1 \sim u_2 \text{, or } u_1 = u_2 \text{ and } v_1 \sim v_2.$$

\begin{construction}[{\cite[Construction 4.10]{HKM2011}}]\label{c4}
Let $G$ be a strongly regular graph with parameters $(v',k',\lambda, \lambda + 1)$. Then $G[K_2]$ is a DDG with parameters $(2v', 2k' + 1,$ $2k', 2\lambda+2, v', 2)$.
\end{construction}

An automorphism of order 2 of a graph is called a \emph{Seidel automorphism} if it interchanges only non-adjacent vertices. Permuting the rows (and not the columns) of the adjacency matrix of a graph according to Seidel automorphism is called \emph{dual Seidel switching} (DSS for short).

\begin{construction}[{\cite[Construction 2]{GHKS2019}}]\label{c5}
Let $\Gamma$ be a DDG obtained with Construction~\ref{c4}. Let $A$ be the adjacency matrix of $\Gamma$, and $P$ be a non-identity permutation matrix of the same size. If $P$ represents a Seidel automorphism, then $PA$ is the adjacency matrix of a DDG with the same parameters as $\Gamma$.
\end{construction}

\begin{proposition}[{\cite[Theorem 2]{GHKS2019}}]\label{p4}
For a proper DDG $\Gamma$ with parameters $(v, k, \lambda_1, \lambda_2,$ $m, n)$, where $\lambda_2 \not\in \{0, 2k - v\}$, the following statements are equivalent.

{\rm (1)} $\lambda_1 = k - 1$.

{\rm (2)} $\Gamma$ comes from Construction 4 or 5.
\end{proposition}

An $m \times m$ matrix $H$ is a \emph{Hadamard matrix} if every entry is $1$ or $-1$ and $HH^\top$~=~$mI_m$. A Hadamard matrix $H$ is called \emph{graphical} if $H$ is symmetric with constant diagonal, and \emph{regular} if all row and column sums are equal.

\begin{construction}[{\cite[Construction 4.9]{HKM2011}}]\label{c6}
Consider a regular graphical Hadamard matrix $H$ of order $l^2 \geqslant 4$ with diagonal entries $-1$ and row sum $l$. The graph with adjacency matrix\\
$A = \begin{bmatrix}
M & N & O\\
N & O & M\\
O & M & N
\end{bmatrix},$
where\\
$M = \displaystyle \frac{1}{2} \begin{bmatrix}
J_{l^2} + H & J_{l^2} + H\\
J_{l^2} + H & J_{l^2} + H
\end{bmatrix},\,
N = \displaystyle \frac{1}{2} \begin{bmatrix}
J_{l^2} + H & J_{l^2} - H\\
J_{l^2} - H & J_{l^2} + H
\end{bmatrix}, \text{ and } O = O_{2l_2}$\\
is a DDG with parameters $(6l^2$, $2l^2 + l$, $l^2 + l$, $(l^2 + l)/2$, $3$, $2l^2)$.
\end{construction}

\subsection{Vertex connectivity}\label{sect2_2}

The \emph{vertex connectivity} $\varkappa(G)$ of a graph $G$ is the minimum number of vertices whose deletion from $G$ disconnects it. Note that for a $k$-regular graph $G$ the inequality $\varkappa(G) \leqslant k$ holds. Let $x$ and $y$ be two vertices in a graph $G$. Two simple paths connecting $x$ and $y$ are called \emph{disjoint} if they have no common vertices different from $x$ and $y$. A set of vertices $S$ \emph{disconnects} $x$ and $y$ if $x$ and $y$ belong to different connected components of the graph obtained from $G$ by deleting $S$. A set $S$ of vertices of a graph $G$ is called \emph{disconnecting} if it disconnects some two of its vertices. The following lemma is known as Menger's theorem.

\begin{lemma}[{\cite[Theorem 5.9]{H1969}}]\label{Menger}
The minimum cardinality of a set disconnecting non-adjacent vertices $x$ and $y$ is equal to the largest number of disjoint paths connecting these vertices.
\end{lemma}

\begin{lemma}[{\cite[Theorem 1]{XY2013}}]\label{lexico}
Let $G_1$ and $G_2$ be two graphs. If $G_1$ is non-complete and connected, then $\varkappa(G1[G2]) = \varkappa(G1) \cdot |V(G2)|$.
\end{lemma}

Below we present known results on the vertex connectivity of Deza graphs.

\begin{lemma}[{\cite[Theorem]{GGK2014}}]\label{dezacon}
Let~$\Gamma$ be a Deza graph obtained from a strongly regular graph~$G$ with non-principal eigenvalues $r$ and $s$ by dual Seidel switching. Then one of the following three cases holds.

{\rm (1)} If $r>2$ and $s<-2$, then the vertex connectivity of~$\Gamma$ is equal to its valency, and a disconnecting set of minimum cardinality is the neighbourhood of some vertex.

{\rm (2)} $s=-2$ and the vertex connectivity of~$\Gamma$ is equal to its valency except for the case when~$G$ is a $3 \times 3$-lattice graph and the vertex connectivity of~$\Gamma$ is~$3$.

{\rm (3)} $r\leq 2$.
\end{lemma}

\begin{lemma}[{\cite[Theorems 1-2]{GP2019}}]\label{dezacon2}
Let~$\Gamma$ be a $k$-regular Deza graph obtained from a strongly regular graph~$G$ with $r = 1$  by dual Seidel switching. Then the vertex connectivity of~$\Gamma$ is equal to its valency except for the case when~$G$ is the complement of $n \times n$-lattice graph $($in this case the vertex connectivity of~$\Gamma$ equals $k-1)$.
\end{lemma}

\section{The vertex connectivity of special classes of DDGs}\label{sect3}

\subsection{DDGs with $\lambda_1 \in \{k-1, k\}$ or $\lambda_2 \in \{0, 2k-v\}$}\label{sect3.1}

\begin{proposition}\label{p0}
The vertex connectivity of a connected DDG $\Gamma$ with parameters $(v, k, \lambda_1, \lambda_2,$ $m, n)$, where $\lambda_2 = 0$, equals $k$.
\end{proposition}

\begin{proof}
By Proposition~\ref{p1}, $\Gamma$ is the incidence graph of a symmetric $2$-$(n,k,\lambda_1)$-design.
Such graphs are distance-regular graphs with diameter $3$ (see \cite[Theorem 1.6.1]{BCN1989}), except for the case when $n = k$ (in this case, $\Gamma$ is the complete bipartite graph $K_{n,n}$ and $\varkappa(\Gamma) = n = k$). Since the vertex connectivity of distance-regular graphs equals~$k$ (see \cite{BK2009}), the statement of the proposition is true.
$\square$ \end{proof}

\begin{proposition}\label{pk}
The vertex connectivity of a connected DDG $\Gamma$ with parameters $(v, k, \lambda_1, \lambda_2,$ $m, n)$, where $\lambda_1 = k$, equals $k$.
\end{proposition}

\begin{proof} 
By Proposition~\ref{p2}, $\Gamma$ can only be obtained by Construction~\ref{c2}. Note that Construction~\ref{c2} can be described as the lexicographic product of a connected $(m,k',\lambda')$-graph and a coclique of size $n$.
Since $(m,k',\lambda')$-graphs with $1 < k' < m$ are strongly regular, their vertex connectivity equals $k'$ (see~\cite{BM1985}). So, by Lemma~\ref{lexico}, $\varkappa(\Gamma) = k'n = k$.
$\square$ \end{proof}

\begin{proposition}\label{p2k_v}
Let $\Gamma$ be a DDG with parameters $(v, k, \lambda_1, \lambda_2,$ $m, n)$, where $\lambda_2 = 2k-v$. Then the following statements hold.

{\rm (1)} If $\lambda_1 \ne k - 1$, then the vertex connectivity of $\Gamma$ equals $k$.

{\rm (2)} If $\lambda_1 = k - 1$, then the vertex connectivity of $\Gamma$ equals $k - 1$.
\end{proposition}

\begin{proof}
By Proposition~\ref{p3}, $\Gamma$ can only be obtained by Construction~\ref{c3}. Recall that the adjacency matrix of $\Gamma$ is $J_v - K_{(m,n)} + diag(A_1, \ldots, A_m)$, where $A_1, \ldots, A_m$ are the adjacency matrices of not necessary connected $(n,k',\lambda')$-graphs $\Gamma_1, \ldots, \Gamma_m$ (possibly non-isomorphic, but having the same parameters). 

Consider two non-adjacent vertices $u_1$ and $u_2$ in $\Gamma$. It follows from Construction~\ref{c3} that $u_1$ and $u_2$ are two non-adjacent vertices in $\Gamma_i$ for some $i \in \{1, \ldots, m\}$. Vertices $u_1$ and $u_2$ have $\varkappa(\Gamma_i)$ disjoint paths connecting them in $\Gamma_i$ and $(m-1)n$ disjoint paths connecting them in the rest of $\Gamma$. In total $u_1$ and $u_2$ have $\varkappa(\Gamma_i) + (m-1)n$ disjoint paths connecting them in $\Gamma$. So, by Lemma~\ref{Menger}, the vertex connectivity of $\Gamma$ equals $min(\varkappa(\Gamma_1), \ldots, \varkappa(\Gamma_m)) + (m-1)n$.

An $(n,k',\lambda')$-graph is disconnected if and only if $k' = 0$ or $k' = 1$. 

If $k' = 0$, then such an $(n,k',\lambda')$-graph is the coclique of size $n$, so $\Gamma$ is the complete $m$-partite graph with parts of size $n$ and $\varkappa(\Gamma) = 0 + (m-1)n = k$. 

If $k' = 1$, then such an $(n,k',\lambda')$-graph is the union of $n/2$ edges, so $\Gamma$ is the complete $m$-partite graph with parts of size $n$ extended with a perfect matching of the complement (see~\cite[Section 4.1]{HKM2011}) and $\varkappa(\Gamma) = 0 + (m-1)n = \lambda_1 = k - 1$. Note that $k' = 1$ is the only case when $\lambda_1 = k - 1$.

If $1 < k' \le n - 2$, then such an $(n,k',\lambda')$-graph is a connected strongly regular graph and its vertex connectivity equals~$k'$. Thus, $\varkappa(\Gamma) = k' + (m-1)n = k$.
$\square$ \end{proof}

\begin{proposition}\label{pk_1}
The vertex connectivity of a DDG $\Gamma$ with parameters $(v, k, \lambda_1, \lambda_2,$ $m, n)$, where $\lambda_1 = k - 1$ and $\lambda_2 \not\in \{0, 2k - v\}$, equals $k - 1$.
\end{proposition}

\begin{proof}
By Proposition~\ref{p4}, for such a DDG we have the following two cases.

\medskip
\noindent{\bf Case 1:} $\Gamma$ is obtained with Construction~\ref{c4}, therefore $\Gamma$ is $G[K_2]$, where $G$ is a strongly regular graph with parameters $(v', k', \lambda, \lambda + 1)$. The vertex connectivity of $G$ equals $k'$, so, by Lemma~\ref{lexico}, $\varkappa(\Gamma) = 2k' = k - 1$.

\medskip
\noindent{\bf Case 2:} $\Gamma$ is obtained with Construction~\ref{c5}. $\Gamma$ can be viewed as follows (see~\cite[Construction 2]{GHKS2019}). Consider $G'[K_2]$, where $G'$ is a Deza graph obtained from a strongly regular graph $G$ with parameters $(v', k', \lambda, \lambda+1)$ by dual Seidel switching with respect to Seidel automorphism~$\varphi$. By the definition, the vertices of $G'[K_2]$ can be viewed as pairs $\{(u, v) : u \in V(G'), v \in V(K_2)\}$. Modify $G'[K_2]$ as follows: for any transposition $(u_1 \, u_2)$ of $\varphi$, take the corresponding two pairs of vertices $(u_1, v_1), (u_1, v_2)$ and $(u_2, v_1), (u_2, v_2)$ in $\Gamma'$, delete the edges $\{(u_1, v_1), (u_1, v_2)\}$ and $\{(u_2, v_1), (u_2, v_2)\}$, and insert the edges $\{(u_1, v_1), (u_2, v_2)\}$ and $\{(u_1, v_2), (u_2, v_1)\}$. The resulting graph is isomorphic to $\Gamma$.

Let $r$ and $s$ be the non-principal eigenvalues of $G$. Consider the cases according to Lemma~\ref{dezacon}.

(1) If $r > 2$ and $s < -2$, then by Lemma~\ref{dezacon}(1) the vertex connectivity of $G'$ equals $k$.

Note that $\Gamma$ can be viewed as two copies of $G'$ connected by additional edges. Also note that since $n = 2$, each entry on the main diagonal of the quotient matrix of $\Gamma$ can only be 0 (if two vertices forming this part are non-adjacent) or 1 (if two vertices forming this part are adjacent). 

Consider two vertices of $\Gamma$ that form a part of the canonical partition. These vertices can be written as $(u_1, v_1)$ and $(u_1, v_2)$, where $v_1 \not= v_2$. The vertices $(u_1, v_1)$ and $(u_1, v_2)$ are adjacent in $\Gamma$ if and only if $u_1$ is fixed by $\varphi$. $\Gamma$ is not walk-regular (see~\cite[Section 5]{GHKS2019}), so, by Lemma~\ref{lwr}, the main diagonal of the quotient matrix is not constant, so it contains both 0 and 1. Thus, there exists a part of the canonical partition of $\Gamma$ consisting of two adjacent vertices. Consider such two vertices. They have $2k'$ common neighbours. If we remove all $2k'$ their common neighbours, we separate the edge formed by these vertices from the rest. So, $\varkappa(\Gamma) \leqslant 2k'$.

Next, we show that, for any pair of non-adjacent vertices in $\Gamma$, there are $2k'$ disjoint paths connecting them. 

There are two types of non-adjacent vertices in $\Gamma$:

(1.1) Vertices $(u_1, v_1)$ and $(u_2, v_2)$, where $u_1$ and $u_2$ are two non-adjacent vertices from $G'$ such that the transposition $(u_1 \, u_2)$ is not in $\varphi$ ($v_1$ can be equal to $v_2$). 

Consider two non-adjacent vertices $(u_1, v_1)$ and $(u_2, v_1)$ from the same copy of $G'$, where $u_1$ and $u_2$ are two non-adjacent vertices from $G'$. Since the vertex connectivity of $G'$ equals $k'$, $(u_1, v_1)$ and $(u_2, v_1)$ have $k'$ disjoint paths connecting them in their copy of $G'$. Also the vertices $(u_1, v_2)$ and $(u_2, v_2)$ have $k'$ disjoint paths connecting them in their copy ($v_1 \not= v_2$). Since the vertices $(u_1, v_1)$ and $(u_1, v_2)$ (as well as $(u_2, v_1)$ and $(u_2, v_2)$) form a part of the canonical partition, they have exactly $k'$ common neighbours in one copy of $G'$ and exactly $k'$ common neighbours in the other copy. Therefore, each of $k'$ disjoint paths connecting $(u_1, v_1)$ and $(u_2, v_1)$ in their copy of $G'$ also connects $(u_1, v_1)$ and $(u_2, v_2)$, $(u_1, v_2)$ and $(u_2, v_1)$, and $(u_1, v_2)$ and $(u_2, v_2)$. A similar argument applies to each of $k'$ disjoint paths connecting $(u_1, v_2)$ and $(u_2, v_2)$ in their copy of $G'$. 

So, any two non-adjacent vertices from the case (1.1) have $k'$ disjoint paths in each copy of $G'$ connecting them, which in total gives $2k'$ disjoint paths connecting them.

(1.2) Vertices $(u_1, v_1)$ and $(u_1, v_2)$, where $u_1$ is moved by $\varphi$, $v_1 \ne v_2$. These vertices form a part of the canonical partition, so they have $2k'$ common neighbours, so they have $2k'$ disjoint paths connecting them.

Any pair of non-adjacent vertices in $\Gamma$ has $2k'$ disjoint paths connecting them, therefore, by Lemma~\ref{Menger}, the vertex connectivity of $\Gamma$ is equal to $2k'$.

(2) If $r \leqslant 2$ and $r$ is not an integer, then $G$ has at most 25 vertices (see~\cite[Conclusion]{GGK2014}). The only strongly regular graphs satisfying this condition are Paley graphs with parameters $(13, 6, 2, 3)$ and $(17, 8, 3, 4)$. By computer calculations using SageMath, these graphs do not have Seidel automorphisms.

(3) If $r = 1$ (or, equivalently, $s = -2$ by Lemma~\ref{teosrg}(1)), then $G$ is a Seidel graphs. There are three Seidel graphs with $\lambda = \mu - 1$: $3\times3$-lattice graph with parameters $(9, 4, 1, 2)$, Petersen graph with parameters $(10, 3, 0, 1)$ and triangular graph $T(5)$ with parameters $(10, 6, 3, 4)$. By computer calculations using SageMath, $T(5)$ does not have Seidel automorphisms, one graph can be obtained by DSS from $3\times3$-lattice graph and one graph can be obtained by DSS from Petersen graph. The vertex connectivity of DDGs obtained from these two graphs equals $k-1$.

(4) If $r=2$ (or, equivalently, $s=-3$ by Lemma~\ref{teosrg}(1)), then there exist 26 SRGs with $\lambda=\mu-1$ (see~\cite{KMP2010}): 15 graphs with parameters $(25, 12, 5, 6)$, 10 graphs with parameters $(26, 10, 3, 4)$ and one graph with parameters $(50, 7, 0, 1)$. By computer calculations using SageMath, the vertex connectivity of DDGs obtained from these graphs equals $k-1$.

Thus, all DDGs obtained by Construction~\ref{c5} have vertex connectivity equal to~$k-1$.
$\square$ \end{proof}

\subsection{DDGs obtained with Construction~\ref{c6}}\label{sect3.2}

Let $\Gamma$ be a DDG with parameters $(6l^2$, $2l^2 + l$, $l^2 + l$, $(l^2 + l)/2$, $3$, $2l^2)$ obtained with Construction~\ref{c6} with positive $l$. Consider the subgraph induced by the first $2l^2$ vertices of $\Gamma$ (in terms of Construction~\ref{c6} this subgraph has adjacency matrix $M$) and the subgraph induced by the last $2l^2$ vertices of $\Gamma$ (in terms of Construction~\ref{c6} this subgraph has adjacency matrix $N$). Denote by $\Gamma_1$ and $\Gamma_2$ the first and the second subgraph, respectively. 

\begin{lemma}\label{l8.0}
$\Gamma_1$ is an $(l^2 + l)$-regular graph and $\Gamma_2$ is an $l^2$-regular graph.
\end{lemma}

\begin{proof}
Consider notations from Construction~\ref{c6}.

Let $x$ and $y$ be the numbers of 1s and $-1$s in each row of $H$, respectively (since $H$ is a regular graphical Hadamard matrix, $x$ and $y$ do not depend on the choice of a row). Then $x - y = l$ and $x + y = l^2$.

Note that the number of 1s in each row of $\frac{1}{2}(J + H)$ is $x$, and the number of 1s in each row of $\frac{1}{2}(J - H)$ is $y$. Thus the number of 1s in each row of $M$ is $2x$, which equals $l^2 + l$, and the number of 1s in each row of $N$ is $x + y$, which equals $l^2$. So $\Gamma_1$ is an $(l^2 + l)$-regular graph and $\Gamma_2$ is an $l^2$-regular graph.
$\square$ \end{proof}

\begin{lemma}\label{l8}
The vertex connectivity of $\Gamma$ is at most $2l^2$.
\end{lemma}

\begin{proof}
Consider notations from Construction~\ref{c6}.

Denote by $A'$ the matrix $\begin{bmatrix}
M & O\\
O & N
\end{bmatrix}$, which can be obtained by removing $2l^2$ rows and columns from the middle of $A$. The matrix $A'$, considered as an adjacency matrix, defines a disconnected graph with two components: the $(l^2 + l)$-regular graph $\Gamma_1$ and the $l^2$-regular graph $\Gamma_2$. Thus, vertex connectivity of $\Gamma$ is at most~$2l^2$.
$\square$ \end{proof}

\begin{lemma}\label{l8.1}
If $\varkappa(\Gamma_1) \ge l^2, \varkappa(\Gamma_2) \ge l^2 - l$ and $\varkappa(\Gamma_1) + \varkappa(\Gamma_2) \ge 2l^2$, then the vertex connectivity of $\Gamma$ equals $2l^2$.
\end{lemma}

\begin{proof}
Denote by $\Gamma_0$ the coclique induced by the $2l^2$ vertices of $\Gamma$ that correspond to the middle rows and columns of $A$. Note that any vertex from $\Gamma_1$ has exactly $l^2$~neighbours in $\Gamma_0$ and any vertex from $\Gamma_2$ has exactly $l^2 + l$ neighbours in $\Gamma_0$.

Let $t_1, t_2$ and $t_0$ be non-negative integers, such that $t_1 + t_2 + t_0 < 2l^2$. If $t_1 \ge \varkappa(\Gamma_1)$ and $t_2 \ge \varkappa(\Gamma_2)$, then $t_1 + t_2 \ge 2l^2$, which contradicts to the inequality above. Denote by $\hat{\Gamma}_1, \hat{\Gamma}_2$ and $\hat{\Gamma}_0$ the graphs obtained by deletion of $t_1, t_2$ and $t_0$ vertices from $\Gamma_1, \Gamma_2$ and $\Gamma_0$, respectively (also denote by $\hat{\Gamma}$ the graph obtained by deletion of $t_1 + t_2 + t_0$ vertices from $\Gamma$).

\medskip
Consider three possible cases:

\medskip
\noindent{\bf Case 1:} $t_1 < \varkappa(\Gamma_1)$ and $t_2 < \varkappa(\Gamma_2)$. Then $\hat{\Gamma}_1$ and $\hat{\Gamma}_2$ are connected. Thus, if there is at least one vertex in $\hat{\Gamma}_0$ that has neighbours in both $\hat{\Gamma}_1$ and $\hat{\Gamma}_2$, then $\hat{\Gamma}$ is connected. Let us show that such a vertex exists.

Since $\Gamma_2$ is an $l^2$-regular graph, we have the inequality $\varkappa(\Gamma_2) \le l^2$ and, consequently, $t_2 < l^2$. Since $l^2$ is less than $l^2 + l$, which is the number of neighbours in $\Gamma_0$ for any vertex from $\Gamma_2$, we conclude that any vertex from $\Gamma_0$ is adjacent to at least one vertex from $\hat{\Gamma}_2$. This means that any vertex from $\hat{\Gamma}_0$ is adjacent to at least one vertex from $\hat{\Gamma}_2$.

Consider two subcases: $t_1 < l^2$ and $t_1 \ge l^2$.

If $t_1 < l^2$, an argument similar to the argument above implies that any vertex from $\hat{\Gamma}_0$ is adjacent to at least one vertex from $\hat{\Gamma}_1$. So there exists a vertex from $\hat{\Gamma}_0$ that has neighbours in both $\hat{\Gamma}_1$ and $\hat{\Gamma}_2$.

If $t_1 \ge l^2$, then the inequalities $t_1 + t_2 + t_0 < 2l^2$ and $t_2 \ge 0$ imply that $t_0$ is less than $l^2$, which is the number of neighbours in $\Gamma_0$ for any vertex from $\Gamma_1$. Therefore any vertex from $\hat{\Gamma}_1$ is adjacent to at least one vertex from $\hat{\Gamma}_0$. So there exists a vertex from $\hat{\Gamma}_0$ that has neighbours in both $\hat{\Gamma}_1$ and $\hat{\Gamma}_2$.

\medskip
\noindent{\bf Case 2:} $t_1 \ge \varkappa(\Gamma_1)$ and $t_2 < \varkappa(\Gamma_2)$. Then $\hat{\Gamma}_1$ is disconnected and $\hat{\Gamma}_2$ is connected. Thus, if, for each connected component of $\hat{\Gamma}_1$, there is at least one vertex in $\hat{\Gamma}_0$ having neighbours in both this component and $\hat{\Gamma}_2$, then $\hat{\Gamma}$ is connected. Let us show that such vertices exist.

Since $t_2 < \varkappa(\Gamma_2) \le l^2$, any vertex from $\hat{\Gamma}_0$ is adjacent to at least one vertex from $\hat{\Gamma}_2$ (see Case~1).

Since $t_1 \ge \varkappa(\Gamma_1) \ge l^2$ and $t_2 \ge 0$, the inequality $t_0 < l^2$ holds. So any vertex from $\hat{\Gamma}_1$ is adjacent to at least one vertex from $\hat{\Gamma}_0$ (see Case~1). So, for each connected component of $\hat{\Gamma}_1$, there is a vertex from $\hat{\Gamma}_0$ having neighbours in both this component and $\hat{\Gamma}_2$.

\medskip
\noindent{\bf Case 3:} $t_1 < \varkappa(\Gamma_1)$ and $t_2 \ge \varkappa(\Gamma_2)$. This case is similar to Case 2.

\medskip
So, the vertex connectivity of $\Gamma$ equals $2l^2$.
$\square$ \end{proof}

\subsection{DDGs with parameters $(6 \cdot 4^t, 2 \cdot 4^t + 2^t, 4^t + 2^t, 2 \cdot 4^{t-1} + 2^{t-1}, 3, 2 \cdot 4^t)$}\label{sect3.3}

If $H_1$ and $H_2$ are Hadamard matrices, then so is the Kronecker product $H_1 \otimes H_2$. Moreover, if $H_1$ and $H_2$ are regular with row sums $l_1$ and $l_2$, respectively, then $H_1 \otimes H_2$ is regular with row sum $l_1 l_2$. Similarly, the Kronecker product of two graphical Hadamard matrices is graphical again.

Consider regular graphical Hadamard matrices $H$ and $H'$, where
$$H = \begin{bmatrix}
-1 & 1 & 1 & 1\\
1 & -1 & 1 & 1\\
1 & 1 & -1 & 1\\
1 & 1 & 1 & -1
\end{bmatrix} \text{ and } H' = \begin{bmatrix}
1 & 1 & 1 & -1\\
1 & 1 & -1 & 1\\
1 & -1 & 1 & 1\\
-1 & 1 & 1 & 1
\end{bmatrix}.$$
Denote by $H_1$ the matrix $H$. For any integer $t$ such that $t > 1$, denote by $H_t$ the Kronecker product $H_{t-1} \otimes H'$. The matrix $H_t$ is a regular graphical Hadamard matrices of order $4^t$ with diagonal entries $-1$ and row sum $2^t$ (see~\cite[Section 10.5.1]{BH2012}).

Applying Construction~\ref{c6} to $H_t$, we obtain a DDG with parameters $(6 \cdot 4^t$, $2 \cdot 4^t + 2^t$, $4^t + 2^t$, $2 \cdot 4^{t-1} + 2^{t-1}$, $3$, $2 \cdot 4^t)$. The smallest example is a DDG with parameters $(24, 10, 6, 3, 3, 8)$ and adjacency matrix
$$\begin{bmatrix}
D & D & D & I & O & O\\
D & D & I & D & O & O\\
D & I & O & O & D & D\\
I & D & O & O & D & D\\
O & O & D & D & D & I\\
O & O & D & D & I & D
\end{bmatrix},$$
where $D = J - I, J = J_4, I = I_4$ and $O = O_4$.

\medskip
By replacing
$$D \rightarrow \begin{bmatrix}
D & D & D & I\\
D & D & I & D\\
D & I & D & D\\
I & D & D & D
\end{bmatrix},\;
I \rightarrow \begin{bmatrix}
I & I & I & D\\
I & I & D & I\\
I & D & I & I\\
D & I & I & I
\end{bmatrix},\;
O \rightarrow \begin{bmatrix}
O & O & O & O\\
O & O & O & O\\
O & O & O & O\\
O & O & O & O
\end{bmatrix},$$
we get a recursive construction for a DDG with parameters $(6 \cdot 4^t$, $2 \cdot 4^t + 2^t$, $4^t + 2^t$, $2 \cdot 4^{t-1} + 2^{t-1}$, $3$, $2 \cdot 4^t)$. Denote by $\Gamma^t$ this DDG.

Consider the subgraph formed by the first $2 \cdot 4^t$ vertices of $\Gamma^t$ (in terms of Construction~\ref{c6} this subgraph has adjacency matrix $M$) and the subgraph formed by the last $2 \cdot 4^t$ vertices of $\Gamma^t$ (in terms of Construction~\ref{c6} this subgraph has adjacency matrix $N$). Let $\Gamma_1^t$ denote the first subgraph and $\Gamma_2^t$ denote the second subgraph.

In the following two lemmas, we present a structural description of $\Gamma_1^t$ and $\Gamma_2^t$, respectively.

\begin{lemma}\label{l9}
The graph $\Gamma_1^t$ is a DDG with parameters $(2 \cdot 4^t, 4^t + 2^t, 4^t + 2^t, 2 \cdot (4^{t-1} + 2^{t-1}),$ $4^t, 2)$.
\end{lemma}

\begin{proof}
The adjacency matrix of $\Gamma_1^t$ is $M$, where the notation is from Construction~\ref{c6}. Set $l := 2^t$, $J := J_{l^2}$ and $I := I_{l^2}$. Then, by $HJ = JH = lJ$ and $H^2 = l^2I$,
\begin{align*}
M^2 &= \displaystyle \frac{1}{4} \begin{bmatrix}
J + H & J + H\\
J + H & J + H
\end{bmatrix}^2\\
&= \displaystyle \frac{1}{4} \begin{bmatrix}
2(J + H)^2 & 2(J + H)^2\\
2(J + H)^2 & 2(J + H)^2
\end{bmatrix}\\
&= \displaystyle \frac{1}{4} \begin{bmatrix}
2(J^2 + HJ + JH + H^2) & 2(J^2 + HJ + JH + H^2)\\
2(J^2 + HJ + JH + H^2) & 2(J^2 + HJ + JH + H^2)
\end{bmatrix}\\
&= \displaystyle \frac{1}{2} \begin{bmatrix}
(l^2 + 2l)J + l^2I & (l^2 + 2l)J + l^2I\\
(l^2 + 2l)J + l^2I & (l^2 + 2l)J + l^2I,
\end{bmatrix}
\end{align*}
which is permutation-equivalent to $\frac{l^2}{2} I_{l^2} \otimes J_2 + \frac{l^2+2l}{2} J_{2l^2}$.

Thus, $\Gamma_1^t$ is a DDG with parameters $(2 \cdot 4^t, 4^t + 2^t, 4^t + 2^t,$ $2 \cdot (4^{t-1} + 2^{t-1}),$ $4^t, 2)$. In particular, any pair of vertices corresponding to equal rows of the adjacency matrix (that is, rows with the same entries) forms a block of size 2 of the canonical partition.
$\square$ \end{proof}

\begin{lemma}\label{l10}
The graph $\Gamma_2^t$ is a DDG with parameters $(2 \cdot 4^t, 4^t, 0,$ $2 \cdot 4^{t-1},$ $4^t, 2)$.
\end{lemma}

\begin{proof}
The adjacency matrix of $\Gamma_2^t$ is $N$, where the notation is from Construction~\ref{c6}. Set $l := 2^t$, $J := J_{l^2}$ and $I := I_{l^2}$. Then, by $HJ = JH = lJ$ and $H^2 = l^2I$,
\begin{align*}
N^2 &= \displaystyle \frac{1}{4} \begin{bmatrix}
J + H & J - H\\
J - H & J + H
\end{bmatrix}^2\\
&= \displaystyle \frac{1}{2} \begin{bmatrix}
(J + H)^2 + (J - H)^2 & (J + H)(J - H) + (J - H)(J + H)\\
(J + H)(J - H) + (J - H)(J + H) & (J + H)^2 + (J - H)^2,
\end{bmatrix}\\
&= \displaystyle \frac{1}{2} \begin{bmatrix}
l^2(J + I) & l^2(J - I)\\
l^2(J - I) & l^2(J + I),
\end{bmatrix}
\end{align*}
which is permutation-equivalent to $\frac{l^2}{2} I_{2l^2} - I_{l^2} \otimes (J_2 - I_2) + \frac{l^2}{2} J_{2l^2}$.

Thus, $\Gamma_2^t$ is a DDG with parameters $(2 \cdot 4^t, 4^t, 0,$ $2 \cdot 4^{t-1},$ $4^t, 2)$. In particular, any pair of vertices corresponding to opposite rows (that is, rows with the opposite entries) of the adjacency matrix forms a block of size 2 of the canonical partition. Note, that since these two vertices correspond to opposite rows, they are adjacent. Thus, $\Gamma_2^t$ has diameter $2$.
$\square$ \end{proof}

There are two known DDGs of diameter $2$ with $\lambda_1 = 0$ (more generally, there are two known strictly Deza graphs with $a = 0$), one on $8$ vertices and one on $32$~vertices. The series $\Gamma_2^t$ covers both cases and gives an infinite series of DDGs of diameter $2$ with $\lambda_1 = 0$ (more generally, strictly Deza graphs with $a = 0$).

\begin{lemma}\label{l11.0}
Let $\Gamma$ be a connected Deza graph with parameters $(v, k, b, a)$ with the second largest eigenvalue $q$. If $\Gamma$ has a disconnecting set of minimum cardinality, that is not the neighbourhood of some vertex, then the following inequality holds: $k - 2q \le b$.
\end{lemma}

\begin{proof}
In this proof we reinterpret the main idea of the proof of \cite[Proposition~5]{GGK2014}. Let $S$ be a disconnecting set of minimum cardinality in $\Gamma$, $|S| = \varkappa(\Gamma) \le k$. Let $A$ and $B$ be the connected components that remain after removing $S$ from $\Gamma$. Assume that $|A| > 1$ and $|B| > 1$, so $S$ is not the neighbourhood of some vertex. Since the spectrum of a disconnected graph is the union of the spectra of connected components, the spectrum $\theta_1 \ge \theta_2 \ge \ldots \ge \theta_{v-|S|}$ of the graph $A \cup B$ is the union of the spectra $\sigma_1 \ge \sigma_2 \ge \ldots \ge \sigma_{|A|}$ and $\omega_1 \ge \omega_2 \ge \ldots \ge \omega_{|B|}$ of the graphs $A$ and $B$, respectively.

The graph $A \cup B$ is an induced subgraph of the graph $\Gamma$ and, by the theorem on spectrum interlacing \cite[Theorem 3.3.1]{BCN1989}, the second largest eigenvalue $\theta_2$ of the graph $A \cup B$ does not exceed the second largest eigenvalue $q$ of the graph $\Gamma$. Evidently, $\theta_2 \ge min(\sigma_1,\omega_1)$. Since the largest eigenvalue in any graph is greater or equal to its mean vertex degree (that is, the arithmetic mean of degrees of its vertices; see \cite[Lemma 3.2.1]{BCN1989}), we can assume without loss of generality that the mean vertex degree in the graph $B$ is at most $q$.

For a vertex $x \in B$, we set $B(x) := B \cap \Gamma(x)$ and $S(x) := S \cap \Gamma(x)$, where $\Gamma(x)$ is the neighbourhood of the vertex $x$ in $\Gamma$. Then $|B(x)|+|S(x)| = k$. Let us estimate the mean vertex degree in the subgraphs $B$ and $S$. Since 
$$\sum\limits_{x \in B} \frac{|B(x)|}{|B|} \le q,$$ 
the inequality
$$\sum\limits_{x \in B} \frac{|S(x)|}{|B|} \ge k - q$$
holds.

Let us estimate the mean number of common neighbours in $S$ for an arbitrary pair of different vertices $x, y \in B$:

\begin{align*}
\sum\limits_{x \in B} \sum\limits_{y \in B \setminus \{x\}} (|S(x)| + |S(y)|) &= \sum\limits_{x \in B} ((|B| - 1)|S(x)| + \sum\limits_{y \in B \setminus \{x\}}|S(y)|)\\
&= (|B| - 1)\sum\limits_{x \in B}|S(x)| + \sum\limits_{x \in B} \sum\limits_{y \in B \setminus \{x\}}|S(y)|\\
&= (|B| - 1)\sum\limits_{x \in B}|S(x)| + (|B| - 1)\sum\limits_{z \in B}|S(z)\\
\frac{\sum\limits_{x \in B} \sum\limits_{y \in B \setminus \{x\}} (|S(x)| + |S(y)|)}{|B|(|B| - 1)} &= \frac{(\sum\limits_{x \in B}|S(x)|)(|B| - 1)\cdot 2}{|B|(|B| - 1)} \ge 2(k - q).
\end{align*}

Since $|B| > 1$, the subgraph $B$ contains a pair of vertices $x$ and $y$ with the property $|S(x)| + |S(y)| \ge 2(k - q)$.

Let $\alpha$, $\beta$ and $\gamma$ be integers such that $\beta = |S(x) \cap S(y)|$, $\alpha+\beta=|S(x)|$ and $\beta+\gamma=|S(y)|$. Then $\alpha+\beta+\gamma \le |S| \le k$. Further, $\alpha+\gamma=|S(x)|+|S(y)|-2\beta$ and $|S(x) \cap S(y)|= \beta \le k-(\alpha+\gamma)=k-(|S(x)|+|S(y)|-2\beta)$. Hence $|S(x)| + |S(y)| \le \beta + k$ and, therefore, $\beta+k \ge 2(k-q)$. Thus, $|S(x) \cap S(y)| = \beta \ge k-2q$.

On the other hand, $|S(x) \cap S(y)| \le b$, which gives the inequality $k - 2q \le b$.
$\square$ \end{proof}

\begin{lemma}\label{l11}
The vertex connectivity of $\Gamma_2^t$ equals $4^t$.
\end{lemma}

\begin{proof}
Let us calculate the spectrum of $\Gamma_2^t$ as a DDG:
\begin{align*}
\{k, \pm \sqrt{k - \lambda_1}, \pm \sqrt{k^2 - \lambda_2v}\} &= \{4^t, \pm \sqrt{4^t - 0},\pm \sqrt{(4^t)^2 - 2 \cdot 4^{t-1} \cdot 2 \cdot 4^t}\}\\
                       &=\{4^t, \pm 2^t, 0\}.
\end{align*}

The second largest eigenvalue of $\Gamma_2^t$ is $2^t$. Considering $\Gamma_2^t$ as a Deza graph with the parameters $(v, k, b, a)$, we get $b = 4^{t - 1}$. So, the inequality from Lemma~\ref{l11.0} becomes $4^t - 2 \cdot 2^t \le 4^{t - 1}$, which holds only for $t = 1$. Thus, for any $t > 1$, the vertex connectivity of $\Gamma_2^t$ equals $k$, where $k = 4^t$. If $t = 1$, then $\Gamma_2^t$ is a DDG with parameters $(8, 4, 0, 2, 4, 2)$. By computer calculations using SageMath, the vertex connectivity of this graph equals $4$.
$\square$ \end{proof}

\begin{theorem}\label{l12}
The vertex connectivity of $\Gamma^t$ equals $2 \cdot 4^t$.
\end{theorem}

\begin{proof}
By Proposition~\ref{pk}, the vertex connectivity of $\Gamma_1^t$ equals $4^t + 2^t$, By Lemma~\ref{l11} the vertex connectivity of $\Gamma_2^t$ equals $4^t$. Thus, by Lemma~\ref{l8.1}, the vertex connectivity of $\Gamma^t$ equals $2 \cdot 4^t$, which is $2^t$~less than the degree of a vertex.
$\square$ \end{proof}

\section{Conclusion}\label{sect4}

Computations in SageMath show that, among connected proper DDGs on at most 39 vertices found in~\cite{PS2022}, there are 32 DDGs with vertex connectivity less than~$k$, where $k$ is the degree of a vertex. For these 32 DDGs, one graph is a DDG with parameters $(24, 10, 6, 3, 3, 8)$ obtained with Construction~\ref{c6}, and the other 31 graphs are DDGs with $\lambda_1 = k - 1$ (this case is described in Propositions~\ref{p2k_v} and \ref{pk_1}).

There are more constructions of regular graphical Hadamard matrices with positive $l$ (see \cite[Section 10.5.1]{BH2012}), so in view of Construction~\ref{c6}, there are more DDGs whose vertex connectivity is less than $k$, where $k$ is the degree of a vertex. In this paper we focused on the smallest graph from Construction~\ref{c6} and its generalisation (in the sense of the recursive construction). We are interested if there exist examples of DDGs, whose vertex connectivity is less than $k$.

There are some approaches for obtaining more general results about the vertex connectivity of DDGs. For example, the approach that was used in \cite{BM1985} and \cite{GGK2014}, or the approach that was used in \cite{BK2009}. Since the spectrum of a DDG is not completely determined by its parameters, the first approach can only be used in specific cases like Lemma~\ref{l11}. The second approach requires more detailed consideration and possible development of new tools to apply it to DDGs. We are interested if general results will be obtained for DDGs, using both known approaches.

\section*{Acknowledgments}
The author is grateful to Sergey Goryainov for valuable comments on the proof of \cite[Proposition~5]{GGK2014} and the text of the paper in general.


\bigskip
\begin{thebibliography}{99}

\bibitem{BCN1989}
A.E.~Brouwer, A.~Cohen, A.~Neumaier, {\it Distance-Regular Graphs}, Springer, 1989.

\bibitem{BH2012}
A.E.~Brouwer, W.H.~Haemers, {\it Spectra of Graphs}, Springer, 2012.

\bibitem{BK2009}
A.E.~Brouwer, J.H.~Koolen, {\it The vertex-connectivity of a distance-regular graph}, European Journal of Combinatorics, {\bf 30(3)} (2009), 668--673.

\bibitem{BM1985} A.E.~Brouwer, D.M.~Mesner, {\it The connectivity of strongly regular graphs}, European Journal of Combinatorics, {\bf 6} (1985), 215--216.

\bibitem{GR2001} C.~Godsil, G.~Royle, {\it Algebraic Graph Theory}, Springer-Verlag, 2001.

\bibitem{CH2014}
D.~Crnkovic, W.H.~Haemers, {\it Walk-regular divisible design graphs}, Designs, Codes and Cryptography, {\bf 72} (2014), 165--175.

\bibitem{EFHHH1999}
M.~Erickson, S.~Fernando, W.H.~Haemers, D.~Hardy, J.~Hemmeter, {\it Deza graphs: A generalization of strongly regular graphs}, Journal of Combinatorial Designs, {\bf 7} (1999), 359--405.

\bibitem{GGK2014}
A.L.~Gavrilyuk, S.~Goryainov, V.V.~Kabanov, {\it On the vertex connectivity of Deza graphs}, Proceedings of the Steklov Institute of Mathematics, {\bf 285(Suppl. 1)} (2014), 68--77.

\bibitem{GHKS2019}
S.~Goryainov, W.H.~Haemers, V.V.~Kabanov, L.~Shalaginov, {\it Deza graphs with parameters $(n,k,k-1,a)$ and $\beta = 1$}, Journal of Combinatorial Designs, {\bf 17(3)} (2019), 188--202.

\bibitem{GP2019}
S.~Goryainov, D.~Panasenko, {\it On vertex connectivity of Deza graphs with parameters of the complements to Seidel graphs}, European Journal of Combinatorics, {\bf 80} (2019), 143--150.

\bibitem{HKM2011}
W.H~Haemers, H.~Kharaghani, M.A.~Meulenberg, {\it Divisible Design Graphs}, Journal of Combinatorial Theory, Series A, {\bf 118} (2011), 978--992.

\bibitem{H1969}
F.~Harary, {\it Graph Theory}, Addison-Wesley, 1969.

\bibitem{KMP2010}
V.V.~Kabanov, A.A.~Makhnev, D.V.~Paduchikh, {\it On strongly regular graphs with eigenvalue $2$ and their extensions},  Doklady Mathematics, {\bf 81} (2010), 268--271.

\bibitem{M2008}
M.A.~Meulenberg, {\it Divisible Design Graphs, Master's thesis}, Tilburg University, 2008.

\bibitem{PS2022}
D.~Panasenko, L.~Shalaginov {\it Classification of divisible design graphs with at most $39$ vertices}, Journal of Combinatorial Designs, {\bf 30(4)} (2022), 205--219.

\bibitem{XY2013}
J.M. Xu, C. Yang {\it Connectivity of lexicographic product and direct product of graphs}, Ars Combinatoria, {\bf 111} (2013). 3--12.

\end{thebibliography}
\end{document}